\documentclass[12pt]{article}
\usepackage{amssymb}
\newtheorem{theorem}{Theorem}[section]
\newtheorem{prop}[theorem]{Propositon}

\newtheorem{preqn}{Question}
\newenvironment{question}{\preqn\rm}{\endpreqn}
\newtheorem{unnumber}{}

\newtheorem{prepf}[unnumber]{Proof}
\newenvironment{proof}{\prepf\rm}{\endprepf}
\newcommand{\Aut}{\mathop{\mathrm{Aut}}}
\newcommand{\qed}{\qquad$\Box$}

\begin{document}
\title{Simplicial complexes defined on groups}
\author{Peter J. Cameron\\University of St Andrews}
\date{}
\maketitle

\begin{abstract}
This paper makes some preliminary observations towards an extension of current
work on graphs defined on groups to simplicial complexes. I define a variety
of simplicial complexes on a group which are preserved by automorphisms of
the group, and in many cases have a relation to familiar graphs on the group.
The ones which seem to reach deepest into the graph structure are two
forms of \emph{independence complex}, and some results on the class of groups
for which these two complexes coincide are given. Other examples are treated
more briefly.


\noindent MOS: 05C25; 20B25
\end{abstract}

\section{Introduction}

The last couple of decades have seen a big upsurge in activity about graphs
defined on groups. In this author's opinion, the most important aspects of
this study are:
\begin{itemize}
\item defining classes of groups by properties of graphs defined on them;
\item finding new results about groups using graphs;
\item constructing beautiful graphs from groups.
\end{itemize}

When I lecture about this topic, it often happens (most recently at the joint
AMS/UMI meeting in Palermo in July 2024) that someone asks, ``What about
simplicial complexes?'' This is a natural question, since there are properties
which are not determined by pairwise relations on the group; but needs to be
handled with some care since, for example, generating sets for a group do not
form a simplicial complex.

A \emph{simplicial complex} $\Delta$ is a downward-closed collection of finite
subsets (called \emph{simplices} or \emph{simplexes}) of a set $X$. We assume
that every singleton of $X$ belongs to $\Delta$. For geometric reasons, a
simplex of cardinality $k$ has \emph{dimension} $k-1$. Thus a point or vertex
of $X$ has dimension $0$, while an edge $\{x,y\}$ has dimension $1$, and a
triangle $\{x,y,z\}$ has dimension~$2$.

The \emph{$k$-skeleton} of $\Delta$ consists of all the subsets of dimension at
most $k$ (thus, cardinality at most $k+1$). Thus, the $1$-skeleton is a graph.

Graph theory is a subject with centuries of history, and many parameters and
properties of graphs have been studied, including cliques and cocliques,
colourings, matchings, cycles, and spectral properties. A glance at a
survey of just one class of graphs defined on groups, the power graph
\cite{kscc}, shows this very clearly.

On the other hand,
simplicial complexes are primarily of interest to topologists: triangulating
a topological space gives a simplicial complex, whose properties converge
to those of the space as the triangulation becomes finer (for nice spaces).
There has been relatively little work on simplicial complexes in their own
right, or combinatorial studies of them. Two of the rare exceptions are
\textit{Euler complexes}~\cite{euler}, which arise in game theory, and
\textit{Ramanujan complexes}~\cite{ramanujan}, a higher dimensional version
of Ramanujan graphs, which have strong expansion properties.

I conclude this introduction with a brief description of some of the graphs
on a group $G$ which play a part in this story. Many of these can be defined
in terms of a subgroup-closed class $\mathcal{C}$ of groups; usually we assume
that $\mathcal{C}$ contains all cyclic groups. In the corresponding graph,
$x$ and $y$ are joined if and only if $\langle x,y\rangle\in\mathcal{C}$.
Examples for $\mathcal{C}$, and the corresponding graphs, include
\begin{itemize}
\item cyclic groups (the \emph{enhanced power graph});
\item abelian groups (the \emph{commuting graph});
\item nilpotent groups (the \emph{nilpotency graph});
\item soluble groups (the \emph{solubility graph}).
\end{itemize}
Other graphs include
\begin{itemize}
\item the \emph{power graph}: $x$ amd $y$ are joined if one is a power of the
other;
\item the \emph{generating graph}: $x$ and $y$ are joined if
$\langle x,y\rangle=G$.
\end{itemize}
The generating graph is null if $G$ cannot be generated by two elements. To
get around this, two further graphs occur in the work of Andrea Lucchini
and coauthors~\cite{flnrd}:
\begin{itemize}
\item the \emph{independence graph}: $x$ and $y$ are joined they are contained
in a minimal (under inclusion) generating set;
\item the \emph{rank graph}: $x$ and $y$ are joined if they are contained in a
generating set of minimum cardinality.
\end{itemize}

To conclude the introduction, here are two general questions which can be 
asked of a simplicial complex on a group:

\begin{question}
Do all maximal simplexes have the same dimension? That is, is the complex
\emph{pure?}\label{q1}
\end{question}

\begin{question}
Do the maximal simplexes generate the group? If not, do those of maximal size
generate the group?\label{q2}
\end{question}

\section{Complexes defined by independence}

I begin with two complexes on a group defined in terms of independence.
A subset $A$ of a group $G$ is \emph{independent} if none of its elements can
be expressed as a word in the other elements and their inverses; or
equivalently, if $a\notin\langle A\setminus\{a\}\rangle$ for all $a\in A$.

\subsection{The independence complex}

The \emph{independence complex} consists of all the independent subsets of $G$.

\begin{prop}
The independence complex is a simplicial complex; its vertices are all the
non-identity elements of $G$. Its $1$-skeleton is the complement of the
power graph of $G$ (with the identity removed).
\end{prop}

\begin{proof}
Suppose that $A$ is independent and $B\subseteq A$. If
$b\in\langle B\setminus b\rangle$, then $b\in\langle A\setminus\{b\}\rangle$,
contradicting the independence of $A$; so $B$ is independent. The second
assertion is clear.

A $2$-element set fails to be independent if and only if one of its elements
is a power of the other; that is, they are joined in the power graph. \qed
\end{proof}

Consider the case where $G$ is an elementary abelian $p$-group for some
prime~$p$ (a direct product of cyclic groups of order $p$). Then $G$ can
be regarded as a vector space over the $p$-element field, so that the subgroups
of $G$ are the subspaces of the vector space; so the complex is pure, and
the maximal simplices generate the group, by basic results of linear algebra.
So for these groups, Questions~\ref{q1} and~\ref{q2} have affirmative answers.

However, for a cyclic group of non-prime order, the complex will not be
pure. For example, if $G$ is cyclic of order $pq$, where $p$ and $q$ are
distinct primes, then a generator is a singleton maximal simplex, while the
pair consisting of elements of orders $p$ and $q$ is a $2$-element maximal
simplex.

Note that, in a cyclic group of prime power order, the power graph is complete,
so the independence complex has dimension $0$.

I mention here an important theorem due to Julius Whiston~\cite{whiston}.

\begin{theorem}
An independent set in the symmetric group $S_n$ has size at most $n-1$, with
equality if and only if it is also a generating set.
\end{theorem}

This has the consequence that the answer to the second part of Question~\ref{q2}
is affirmative for the symmetric groups. Note that the independent sets in $S_n$
of cardinality $n-1$ have all been determined by Cameron and Cara~\cite{cc}.

\subsection{The strong independence complex}

The concept of independence in a group has been investigated by a number of
authors. However, I know no literature on the next concept.

A subset $A$ of a group $G$ is called \emph{strongly independent} if no
subgroup of $G$ containing $A$ has fewer than $|A|$ generators.
The \emph{strong independence complex} of $G$ is the complex whose simplices
are the strongly independent sets.

\begin{prop}
The strong independence complex is a simplicial complex, whose vertices are the
non-identity elements of $G$. Its $1$-skeleton is the complement of the
enhanced power graph of $G$.
\end{prop}

\begin{proof}
Let $A$ be strongly independent and $B\subseteq A$. Suppose that $B$ is not
strongly independent; so there exists a subgroup $H$ containing $B$ which
is generated by a set $C$ with $|C|<|B|$. Let $K$ be the subgroup generated
by $(A\setminus B)\cup C$. Then $K$ contains $B$, and hence $A$; and
the generating set of $K$ is smaller than $A$, contradicting strong
independence of $A$.

It is clear that a singleton $\{x\}$ is strongly independent if and only if
$x\ne1$.

For the last part, suppose that $x$ and $y$ are adjacent in the enhanced power
graph, so that $\langle x,y\rangle=\langle z\rangle$ for some $z$, so
$\{x,y\}$ is not strongly independent. Conversely, if $\{x,y\}$ is not
strongly independent, then $x,y\in\langle z\rangle$, so $x$ and $y$ are 
adjacent in the enhanced power graph. \qed
\end{proof}

In an elementary abelian $2$-group, once again linear algebra shows that
strong independence coincides with linear independence, so maximal simplices
in the strong independence complex are bases; any two of them have the same
cardinality, and any one generates the group.

This prompts a further question:

\begin{question}
For which groups do the notions of independence and strong independence
coincide?\label{q3}
\end{question}

\medskip

Note that this class of groups is subgroup-closed.

I can give a necessary condition. An \emph{EPPO} group is a group
in which all elements have prime power order. After earlier results by 
Higman and Suzuki, the EPPO groups were classified by Brandl~\cite{brandl}.
A recent and possibly more accessible account is given in the survey~\cite{cm}.

\begin{theorem}
The class of finite groups in which independence and strong independence
coincide is subgroup-closed, and is contained in the class of EPPO groups.
\end{theorem}

\begin{proof}
The first statement is clear.

The $1$-skeletons of the independence and strong independence complexes
coincide, and hence so do their complements, the power graph and the
enhanced power graph. Groups for which these two graphs are equal are
precisely the EPPO groups, see~\cite{aacns}.

The proof is straightforward. The enhanced power graph of any cyclic group is
complete; but the power graph of $C_n$ is complete if and only if
$n$ is a prime power. (If $x$ has order $pq$, where $p$ and $q$ are distinct
primes, then $x^p$ and $x^q$ are nonadjacent in the power graph.) \qed
\end{proof}

This condition is not sufficent. One of the EPPO groups in Brandl's list is
the Suzuki group $\mathrm{Sz}(8)$. The Sylow $2$-subgroup has a minimal
generating set of size $3$, but the whole group can be generated by two
elements. We will see another example shortly.

In the other direction, we have:

\begin{theorem}
In an abelian $p$-group, for $p$ prime, independence and strong independence
coincide.
\end{theorem}

\begin{proof}
Let $G$ be an abelian $p$-group, written additively. Suppose that 
\[G=C(p^{a_1})\oplus C(p^{a_2})\oplus\cdots\oplus C(p^{a_r}),\]
where $a_1\ge a_2\ge\cdots\ge a_r$.

First we prove that an independent set has cardinality at most $r$.
The proof is by induction on $r$. If $r=1$, so that $G$ is cyclic, then
its subgroups form a chain; so given any two elements, one lies in the
subgroup generated by the other. So there is no independent set of size
greater than $1$. Thus, the induction starts.

Suppose that it holds for $r-1$. Let $X=\{x_1,x_2,\ldots,x_k\}$. Each $x_i$
can be written as a linear combination of the generators $z_1,\ldots,z_r$ of
the cyclic factors of $G$. If $z_1$ never occurs, then $X$ is contained in 
a sum of $r-1$ cyclic groups, so $|X|=k\le r-1$, and we are done.

So suppose that $z_1$ does occur in such expressions. We choose an element
of $X$ whose component in the first factor has largest possible order;
without loss of generality, this is $x_1=tz_1+\cdots$. Now we use Gaussian
elimination. The first components of all the other elements of $X$ lie in
the cyclic group generated by $tz_1$; so, replacing $x_i$ by
$y_i=x_i-u_ix_1$ for suitable $u_i$, we see that the first components of
$y_1,\ldots,y_k$ are all equal to zero. Hence $\{y_2,\ldots,y_k\}$ is
contained in the sum of $r-1$ cyclic groups. But Gaussian elimination
preserves independence; so, by the inductive hypothesis, $k-1\le r-1$, whence
$k\le r$, as required.

Thus any independent set in an abelian $p$-group is not contained in a subgroup
where the number of generators is less than the size of the set, and hence is
strongly independent. \qed
\end{proof}

Both conditions in the theorem are necessary:
\begin{itemize}
\item If $G$ is abelian but not a $p$-group, then it contains an element of
order $pq$ for distinct primes $p$ and $q$, and so is not an EPPO group.
\item There are non-abelian $p$-groups which fail the condition. For
example, $C_p\wr C_p$ for $p$ an odd prime, or $C_2\wr C_4$, are
$2$-generated but contain independent sets of sizes $3$ and $4$ respectively.
However, we note that some non-abelian $p$-groups, for example dihedral
groups, satisfy the condition.
\end{itemize}

There are various other groups where independence and strong independence
coincide; for example, the non-abelian group of order $pq$, where $p$ and
$q$ are primes with $q\mid p-1$, and the alternating group $A_4$.
The problem of determining these groups should be approachable.

For example, we can show that no simple group can occur.
For Brandl's list~\cite{brandl} gives the following possibilities:
\begin{itemize}
\item $G=\mathrm{PSL}(2,q)$, for $q=4,7,8,9,17$;
\item $G=\mathrm{Sz}(q)$, for $q=8,32$;
\item $G=\mathrm{PSL}(3,4)$.
\end{itemize}
All are $2$-generated. It helps to observe that $S_4$ does not have the
property, since it is $2$-generated but the standard Moore--Coxeter generators
$(1,2)$, $(2,3)$, $(3,4)$ are independent. Now subgroup
closure excludes $\mathrm{PSL}(2,q)$ for $q=7$, $9$ and $17$.
Also $A_5$ is $2$-generated but the elements $(1,2,3)$, $(1,2,4)$, $(1,2,5)$
are independent; $\mathrm{PSL}(2,8)$ is $2$-generated but contains $(C_2)^3$.
The Suzuki group $\mathrm{Sz}(2,8)$ has already been dealt with, and similar
arguments deal with the other Suzuki group. Finally, $\mathrm{PSL}(3,4)$
contains $\mathrm{PSL}(3,2)\cong\mathrm{PSL}(2,7)$.

\section{Further examples}

Let $\mathcal{C}$ be a subgroup-closed class of groups. Then the collection
of subsets $S$ of a group $G$ for which $\langle S\rangle\in\mathcal{C}$ is
a simplicial complex, for which every element of $G$ is a point if $\mathcal{C}$
contains all cyclic groups. The $1$-skeleton of this complex is the graph
defined in terms of $\mathcal{C}$ in the introduction.

The maximal simplexes are precisely the maximal $\mathcal{C}$-subgroups of $G$.
So no maximal simplex generates $G$ unless $G\in\mathcal{C}$, in which case
$G$ is a simplex.

\medskip

The generating sets do not form a simplicial complex, but their
complements do: this is the \emph{non-generating complex}.

\medskip

Further simplicial complexes can be formed by imposing two different conditions
on the simplices. Examples include
\begin{itemize}
\item the commuting independence complex, where a simplex is an independent set
whose elements commute pairwise;
\item the non-generating independence complex, where a simplex is an
independent set which does not generate the group.
\end{itemize}

These somewhat resemble various difference graphs which have been studied
on groups, such as the difference beween the enhanced power graph and the
power graph~\cite{bcdd}. I mention here also the work of Saul Freedman on
the non-commuting, non-generating graph~\cite{freedman}, and the work of
Freedman \emph{et al.} on groups for which the power graph is equal to the
complement of the independence graph, or the enhanced power graph is the
complement of the rank graph~\cite{flnrd}.

\medskip

Various suggestions for simplicial complexes have, roughly speaking, the form:
$\{x_1,\ldots,x_k\}$ is a simplex if and only if $w(x_1,\ldots,x_k)=1$,
for some fixed word $w$. There are a couple of problems with this definition
as stated. First, it depends on the order of $x_1,\ldots,x_k$; so, instead,
we have to say, $\{x_1,\ldots,x_k\}$ is a simplex if and only if there is a
permutation $\pi\in S_k$ such that $w(x_{1\pi},\ldots,x_{k\pi})=1$. Also,
we have to insist that $x_1,\ldots,x_k$ are all distinct.

Then the problem arises of extending the definition to other dimensions.
Going up is easy: $\{x_1,\ldots,x_m\}$ is a simplex (for $m\ge k$) if and
only if all of its $k$-subsets are simplices. For going down, we have to say,
$\{x_1,\ldots,x_l\}$ is a simplex, for $l\le k$, if and only if there exist
$x_{l+1},\ldots,x_k$ such that $\{x_1,\ldots,x_l,x_{l+1},\ldots,x_k\}$ is
a simplex.

A simple example with $k=2$ is the commuting complex described earlier,
with $w=[x,y]=x^{-1}y^{-1}xy$.  Examples with $k=3$ include $w(x,y,z)=xyz$, or
$z(x,y,z)=[[x,y],z]$. In these cases, some partial symmetries already hold,
so we don't need to apply the condition for all $\pi\in S_k$. We have
\[xyz=1\Rightarrow yzx=zxy=1\]
and
\[[[x,y,z]=1\Rightarrow[[y,x],z]=1.\]

I will not discuss these examples further.

\section{The Gruenberg--Kegel complex}

The Gruenberg--Kegel graph of a finite group $G$ is the graph whose vertex set
is the set $\Pi$ of prime divisors of $|G|$, with $p$ and $q$ joined if
$G$ contains an element of order $pq$. It was introduced by Gruenberg and
Kegel in connection with the integral group ring of $G$. They determined
the groups for which the Gruenberg--Kegel graph is disconnected, but did not
publish their result. It was published by Gruenberg's student
Williams~\cite{williams}. These graphs have been the subject of a lot of
further research.

Define the \emph{Gruenberg--Kegel complex} $\mathrm{GKC}(G)$ of $G$ by the
rule that the vertex set is $\Pi$ (as above), and a subset $\Gamma$ is a
simplex if and only if $G$ contains an element whose order is the product
of the primes in $\Gamma$.

\begin{prop} Let $G$ be a finite group.
\begin{itemize}
\item The $1$-skeleton of $\mathrm{GKC}(G)$ is the Gruenberg--Kegel graph
of $G$.
\item If $G$ is nilpotent, then $\mathrm{GKC}(G)$ is a simplex.
\item For every group $G$, there is a positive integer $d(G)$ with the
property that $\mathrm{GKC}(G^n)$ is a simplex if and only if $n\ge d(G)$.
\end{itemize}
\end{prop}

\begin{proof}
The first statement is clear. For the second, we note that elements of distinct
prime orders in a nilpotent group commute. The third is easy to see; in
fact $d(G)$ is the smallest integer $k$ for which $\mathrm{GKC}(G)$ is the
union of $k$ simplices. (We can assume that the union is disjoint. Then take
$x_i$ to be an element in the $i$th direct factor whose order is the product
of primes in the $i$th simplex. The elements in different factors commute,
so the product of the $x_i$ is divisible by every prime in $\Pi$.) \qed
\end{proof}

For a recent survey on the Gruenberg--Kegel graph, see~\cite{cm}.

\section{Homology and representations}

A simplicial complex possesses homology (and cohomology) groups. For complexes
defined on $G$ and invariant under $\Aut(G)$, these will be $\Aut(G)$-modules.

Homological algebra has been around for a while, so perhaps little new can be
found about homology groups. On the other hand, persistent homology is an
important new tool in data analysis, so maybe homological tools have a role
to play here.

\begin{question}
Which representations of $\Aut(G)$ are realised on homology groups
associated with natural simplicial complexes for $G$?
\end{question}

\end{document}